\documentclass[11pt]{amsart}
\usepackage{amssymb}

\usepackage[latin1]{inputenc}
\usepackage{mathptmx}

\newcommand{\mm}{\mathfrak m}

\newcommand{\pp}{\mathfrak p}
\newcommand{\qq}{\mathfrak q}
\newcommand{\ka}{\mathfrak a}
\newcommand{\kb}{\mathfrak b}

\newcommand{\N}{\mathbb{N}}

\newcommand{\Pc}{\mathcal{P}}

\newcommand{\ab}{{\bf a}}
\newcommand{\bb}{{\bf b}}

\newcommand{\xb}{{\bf x}}

\DeclareMathOperator{\height}{ht}

\newcommand{\nil}{\operatorname{Nil}}
\newcommand{\Hom}{\operatorname{Hom}}

\newcommand{\spec}{\operatorname{Spec}}

\newcommand{\Supp}{\operatorname{Supp}}
\newcommand{\depth}{\operatorname{depth}}
\newcommand{\grade}{\operatorname{grade}}

\newcommand{\Min}{\operatorname{Min}}
\newcommand{\Max}{\operatorname{Max}}

\DeclareMathOperator{\pnt}{\raise 0.5mm \hbox{\large\bf.}}

\DeclareMathOperator{\kgr}{K.grade}

\DeclareMathOperator{\chgr}{\check{C}.grade}

\def\+#1{\relax\ifmmode\if\noexpand #1\relax \mathop{\kern
    0pt^+{#1}}\nolimits\else \kern 0pt^+\!#1 \fi\else$^*$#1\fi}

\newtheorem{thm}{\bf Theorem}[section]
\newtheorem{lem}[thm]{\bf Lemma}
\newtheorem{cor}[thm]{\bf Corollary}
\newtheorem{prop}[thm]{\bf Proposition}

\theoremstyle{definition}

\newtheorem{ex}[thm]{\bf Example}

\theoremstyle{plain}
\newtheorem*{thm*}{Theorem}

\title[Cohen-Macaulay]{Cohen-Macaulay properties under
    the amalgamated construction}

\author{Y. Azimi, P. Sahandi, and N. Shirmohammadi}
\address{Department of Mathematics, University of Tabriz, Tabriz, Iran.}
\email{u.azimi@tabrizu.ac.ir}
\address{Department of Mathematics, University of Tabriz, Tabriz, Iran.}
\email{sahandi@ipm.ir}
\address{Department of Mathematics, University of Tabriz, Tabriz, Iran.}
\email{nshirmohammadi@yahoo.com, shirmohammadi@tabrizu.ac.ir}


\begin{document}

\begin{abstract}
Let $A$ and $B$ be commutative rings with unity, $f:A\to B$ a ring
homomorphism and $J$ an ideal of $B$. Then the subring
$A\bowtie^fJ:=\{(a,f(a)+j)|a\in A$ and $j\in J\}$ of $A\times B$ is
called the amalgamation of $A$ with $B$ along $J$ with respect to
$f$. In this paper, we study the property of Cohen-Macaulay in the
sense of ideals which was introduced by Asgharzadeh and Tousi, a
general notion of the usual Cohen-Macaulay property (in the
Noetherian case), on the ring $A\bowtie^fJ$. Among other things, we
obtain a generalization of the well-known result that when the
Nagata's idealization is Cohen-Macaulay.
\end{abstract}

\maketitle

\section{Introduction}
\label{intro} The theory of Cohen-Macaulay rings is a major area of
study in commutative algebra and algebraic geometry. From the
appearance of the notion of Cohen-Macaulayness, this notion admits a
rich theory in commutative \emph{Noetherian} rings. There have been
attempts to extend this notion to commutative \emph{non-Noetherian}
rings, since Glaz raised the question that whether there exists a
generalization of the notion of Cohen-Macaulayness with certain
desirable properties to non-Noetherian rings \cite{G0}, \cite{G}. In
order to provide an answer to the question of Glaz \cite[Page
220]{G}, recently several notions of Cohen-Macaulayness for
non-Noetherian rings and modules were introduced in \cite{H},
\cite{HM}, and \cite{AT}. Among those is the Cohen-Macaulay in the
sense of $\mathcal{A}$, introduced by Asgharzadeh and Tousi
\cite{AT}, where $\mathcal{A}$ is a non-empty subclass of ideals of
a commutative ring (the definition will be given later in Section
2).

In \cite{DFF} and \cite{DFF2}, D'Anna, Finocchiaro, and Fontana have
introduced the following new ring construction. Let $A$ and $B$ be
commutative rings with unity, let $J$ be an ideal of $B$ and let
$f:A\to B$ be a ring homomorphism. The \emph{amalgamation of $A$
with $B$ along $J$ with respect to $f$} is the following subring
$$A\bowtie^fJ:=\{(a,f(a)+j)|a\in A\text{ and }j\in J\}$$ of $A\times
B$. This construction generalizes the amalgamated duplication of a
ring along an ideal (introduced and studied in \cite{D}, \cite{DF}).
Moreover, several classical constructions such as the Nagata's
idealization (cf. \cite[page 2]{Na}, \cite[Chapter VI, Section
25]{Hu}), the $A + XB[X]$ and the $A+XB[[X]]$ constructions can be
studied as particular cases of this new construction (see
\cite[Examples 2.5 and 2.6]{DFF}).

Below, we review briefly some known results about the behavior of
Cohen-Macaulayness under the amalgamated construction and its
particular cases.

Let $M$ be an $A$-module. In 1955, Nagata introduced a ring
extension of $A$ called the \emph{trivial extension} of $A$ by $M$
(or the \emph{idealization} of $M$ in $A$), denoted here by
$A\ltimes M$. Now, assume that $A$ is Noetherian local and that $M$
is finitely generated. It is well known that the trivial extension
$A\ltimes M$ is Cohen-Macaulay if and only if $A$ is Cohen-Macaulay
and $M$ is maximal Cohen-Macaulay, see \cite[Corollary 4.14]{AW}.

Let $A$ be a Noetherian local ring and $I$ be an ideal of $A$.
Consider the amalgamated duplication $A\bowtie I:=\{(a,a+i)|a\in
A\text{ and }i\in I\}$ as in \cite{D}, \cite{DF}. The properties of
being Cohen-Macaulay, generalized Cohen-Macaulay, Gorenstein,
quasi-Gorenstein, $(S_n)$, $(R_n)$ and normality under the
construction of amalgamated duplication were studied further in many
research papers such as \cite{D}, \cite{DFF1}, \cite{BSTY}, and
\cite{SSh}.

In \cite{DFF1}, under the condition that $A$ is Cohen-Macaulay
(Noetherian local) and $J$ is a finitely generated $A$-module, it is
observed that $A\bowtie^f J$ is a Cohen-Macaulay ring if and only if
it is a Cohen-Macaulay $A$-module if and only if $J$ is a maximal
Cohen-Macaulay module. Then, in \cite{SSS}, assuming $(A,\mm)$ is
Noetherian local, $J$ is contained in the Jacobson radical of $B$
such that $\depth_AJ<\infty$ and that $f^{-1}(\qq)\neq\mm$, for each
$\qq\in\spec(B)\backslash V(J)$, it is shown that $A\bowtie^f J$ is
Cohen-Macaulay if and only if $A$ is Cohen-Macaulay and $J$ is a big
Cohen-Macaulay $A$-module (i.e. $\depth_AJ=\dim A$).

The next natural step is to seek when the amalgamated algebra
$A\bowtie^fJ$ is Cohen-Macaulay without Noetherian assumption.

In this paper, we investigate the property of Cohen-Macaulay in the
sense of ideals (resp. maximal ideals, finitely generated ideals) on
the amalgamation. More precisely, in Section 2, we recall some
essential definitions and results on which we base our approach. In
Section 3, we fix our notation and give some elementary results on
the behavior of the Koszul grade with respect to amalgamation. In
Section 4, we classify some necessary and sufficient conditions for
the amalgamated algebra $A\bowtie^f J$ to be Cohen-Macaulay in the
sense of ideals (resp. maximal ideals, finitely generated ideals)
(Theorems \ref{m}, \ref{jzirenil} and \ref{gd}). Among the
applications of our results are the classification of when the
trivial extension $A\ltimes M$ and the amalgamated duplication
$A\bowtie I$ are Cohen-Macaulay in the sense of ideals (Corollaries
\ref{t} and \ref{d}).

\section{Preliminaries}

To facilitate the reading of the paper, we recall in this section
some preliminary definitions and properties to be used later.

Let $\kb$ be a finitely generated ideal of a commutative ring $A$
and $M$ be an $A$-module. Assume that $\kb$ is generated by the
sequence $\xb = x_1,\ldots ,x_\ell$. We denote the Koszul complex
related to $\xb$ by $\mathbb{K}_{\bullet}(\xb)$. The \emph{Koszul
grade of $\kb$ on $M$} is defined by
$$\kgr_A (\kb,M):= \inf \{i\in \N\cup \{0\}|H^i( \Hom_A(\mathbb{K}_{\bullet}(\xb),M))\neq 0\}.$$
It follows from \cite[Corollary 1.6.22]{BH98} and \cite[Proposition
1.6.10(d)]{BH98} that this does not depend on the choice of
generating sets of $\kb$.

Let $\ka$ be an arbitrary ideal of $A$. One can then define the
Koszul grade of $\ka$ on $M$ by setting
$$\kgr_A (\ka, M) := \sup \{ \kgr_A (\kb, M)| \kb\text{ is a finitely generated subideal of }\ka\}.$$
In view of \cite[Proposition 9.1.2(f)]{BH98}, this definition
coincides with the original one for finitely generated ideals. In
particular, when $(A,\mm)$ is local Noetherian, $\depth_AM$ was
defined by $\kgr_A (\mm, M)$ in \cite[Section 9.1]{BH98}.

The \emph{\v{C}ech grade of $\kb$ on $M$} is defined by
$$\chgr_A (\kb, M):= \inf \{i\in \N\cup \{0\}| H^i_{\xb} (M) \neq 0\}.$$
Here $H^i_{\xb}(M)$ denotes the $i$-th cohomology of the
$\check{C}ech$ complex of $M$ related to $\xb$. It follows from
\cite[Proposition 2.1(e)]{HM} that $H^i_{\xb}(M)$ is independent of
the choice of sequence of generators for $\kb$. One can then define
$$\chgr_A (\ka, M) := \sup \{ \chgr_A (\kb, M)| \kb\text{ is a finitely generated subideal of }\ka\}.$$
By virtue of \cite[Proposition 2.7]{HM}, one has $\chgr_A (\ka,
M)=\kgr_A (\ka,M)$.

Let $\pp$ a prime ideal of $R$. By $\height_M\pp$, we mean the Krull
dimension of the $R_\pp$-module $M_\pp$. Also,
$$\height_M\ka:=\inf \{\height_M\pp|\pp \in \Supp_A(M) \cap V(\ka) \}.$$

Let $\mathcal{A}$ be a non-empty subclass of the class of all ideals
of the ring $A$ and $M$ be an $A$-module. We say that $M$ is
\emph{Cohen-Macaulay in the sense of $\mathcal{A}$} if
$\height_M(\ka)=\kgr_A(\ka,M)$ for all ideals $\ka$ in
$\mathcal{A}$, see \cite[Definition 3.1]{AT}. The classes we are
interested in are the class of all maximal ideals, the class of all
ideals and the class of all finitely generated ideals. Assume that
$A$ is Noetherian. It is well-known that $A$ is Cohen-Macaulay (in
the sense of the original definition in the Noetherian setting) if
and only if it is Cohen-Macaulay in the sense of ideals (resp.
maximal ideals, finitely generated ideals) see \cite[Corollary
2.1.4]{BH98}.

\section{The Koszul grade on amalgamation}

Let us fix some notation which we shall use frequently throughout
the paper: $A,\  B$ are two commutative rings with unity, $f:A\to B$
is a ring homomorphism, and $J$ denotes an ideal of $B$. So that $J$
is an $A$-module via the homomorphism $f$. In the sequel, we
consider the contraction and extension with respect to the natural
embedding $\iota _A: A\to A\bowtie^f J$ defined by $\iota _A
(x)=(x,f(x))$, for every $x\in A$. In particular, for every ideal
$\ka$ of $A$, $\ka^e$  means $\ka(A\bowtie^f J)$.

This section is devoted to prove some lemmas on the behavior of the
Koszul grade on amalgamation. These lemmas provide the key for some
crucial arguments later in this paper. In the proof of the next
lemma, we use $H_{i}(\xb,M)$ to denote the $i$th Koszul homology of
an $A$-module $M$ with respect to a finite sequence $\xb\subset A$.

\begin{lem}\label{kgr1}
Let the notation and hypotheses be as in the beginning of this
section. Then
\begin{enumerate}
  \item for any finitely generated ideal $\kb$ of $A$, one has the
  equality
$$\kgr_{A\bowtie^f J} (\kb^e,A\bowtie^f J)= \min
\{\kgr_A(\kb,A),\kgr_A(\kb,J)\}.$$
  \item for any ideal $\ka$ of $A$, one has the inequality
  $$\kgr_{A\bowtie^f J} (\ka^e,A\bowtie^f J)\le \min
  \{\kgr_A(\ka,A),\kgr_A(\ka,J)\}.$$
\end{enumerate}
\end{lem}
\begin{proof}
Assume that $\kb$ is a finitely generated ideal of $A$ and that
$\kb$ is generated by a finite sequence $\xb$ of length $\ell$.
Then, using \cite[Proposition 2.2(iv)]{AT} together with
\cite[Proposition 2.7]{HM}, we have
\begin{align*}
        & \kgr_{A\bowtie^f J} (\kb^e,A\bowtie^f J)\\
        =&\kgr_A (\kb,A\bowtie^f J)\\
        =&\sup \{k\ge 0 |  H_{\ell -i}(\xb, A\bowtie^f J)=0\ \text{for all}\ i<k  \}\\
        =&\sup \{k\ge 0 |  H_{\ell -i}(\xb, A)\oplus H_{\ell -i}(\xb, J)=0\ \text{for all}\ i<k  \}\\
        =&\min \{\kgr_A(\kb,A),\kgr_A(\kb,J)\}.
\end{align*}
For the third equality, one notices that the amalgamation
$A\bowtie^f J$, as an $A$-module, is isomorphic to the direct sum of
$A\oplus J$ using \cite[Lemma 2.3(4)]{DFF}. This proves (1). To
obtain (2), assume that $\ka$ is an ideal of $A$. Let $\Sigma$ be
the class of all finitely generated subideals of $\ka$. It follows
from the definition that
\begin{align*}
        &\kgr_A (\ka,A\bowtie^f J)\\
        =&\sup \{ \kgr_A (\kb,A\bowtie^f J) | \kb \in \Sigma \}\\
        =& \sup \{ \min \{\kgr_A(\kb,A),\kgr_A(\kb,J)\} | \kb \in \Sigma \}\\
        \le& \min \{\sup \{\kgr_A(\kb,A)| \kb \in \Sigma \}, \sup \{\kgr_A(\kb,J)| \kb \in \Sigma \}\}\\
        =&\min \{\kgr_A(\ka,A),\kgr_A(\ka,J)\}.
\end{align*}
Again, using this in conjunction with \cite[Proposition
2.2(iv)]{AT}, one deduces that
\begin{align*}
        \kgr_{A\bowtie^f J}(\ka^e,A\bowtie^f J)
        &=\kgr_A(\ka,A\bowtie^f J)\\
        &\le \min \{\kgr_A(\ka,A),\kgr_A(\ka,J)\}.
\end{align*}
\end{proof}

\begin{lem}\label{kgr}
Assume that $A$ is  Cohen-Macaulay in the sense of (finitely
generated) ideals and $\kgr_A(\ka,J)\ge \height \ka$ for every
(finitely generated) ideal $\ka$ of $A$. Then
$$\kgr_{A\bowtie^f J} (\ka^e,A\bowtie^f J)=\kgr_A(\ka,A)\le \kgr_A(\ka,J)$$
for any (finitely generated) ideal $\ka$ of $A$.
\end{lem}
\begin{proof}
Assume that $\ka$ is a (finitely generated) ideal of $A$ and let
$\Sigma$ be the class of all finitely generated subideals of $\ka$.
Then, as in the proof of Lemma \ref{kgr1}, again, using
\cite[Proposition 2.2(iv)]{AT}, we have
\begin{align*}
        &\kgr_{A\bowtie^f J}(\ka^e,A\bowtie^f J)\\
        =&\kgr_A (\ka,A\bowtie^f J)\\
        =&\sup \{ \kgr_A (\kb,A\bowtie^f J) | \kb \in \Sigma \}\\
        =&\sup \{ \min \{\kgr_A (\kb,A),\kgr_A (\kb,J)\} | \kb \in \Sigma \}\\
        =&\sup \{ \kgr_A (\kb,A) | \kb \in \Sigma \}\\
        =&\kgr_A (\ka,A).
\end{align*}
The forth equality follows from \cite[Lemma 3.2]{AT} and our
assumption. This completes the proof.
\end{proof}

The following lemma is a slight modification of \cite[Lemma
3.2]{AT}.

\begin{lem}\label{tamime 3.2}
Let $\ka$ be an ideal of $A$ and $M$ be an $A$-module.
\begin{enumerate}
  \item Let $A$ be quasi-local with the maximal ideal $\mm$. If $\kgr_A
(\mm, M)<\infty$, then $\kgr_A (\mm, M)\leq\dim A$.
  \item If, for every minimal prime ideal $\pp$ over $\ka$, $\kgr_A(\pp
R_\pp, M_\pp)< \infty$ (e.g. when $M$ is finitely generated), then
$\kgr_A(\ka, M)\le \height\ka.$
\end{enumerate}
\end{lem}
\begin{proof}
(1) Using \cite[Proposition 2.7]{HM}, it is enough for us to show
that $\chgr_A (\mm, M)\leq\dim A$. In order to prove this, assume
that $\dim A<\infty$ and let ${\bf x}$ be a finite sequence of
elements in $\mm$. It follows from \cite[Proposition 2.4]{HM} that
$\chgr_A({\bf x},M)\le \dim A$. Therefore $\chgr_A(\mm, M)\leq\dim
A$. (2) Notice, by \cite[Proposition 2.2(iii)]{AT}, that
$\kgr_A(\ka, M)< \infty$. Then, by \cite[Proposition 2.2(ii) and
(iii)]{AT}, one may assume that $A$ is quasi-local with the maximal
ideal $\mm$. Now (1) completes the proof.
\end{proof}

\section{Main results}

Assume that $A$ is Noetherian local, and that $J$ is contained in
the Jacobson radical of $B$ and it is a finitely generated
$A$-module. Recall that a finitely generated module $M$ over $A$ is
called a \emph{maximal Cohen-Macaulay $A$-module} if $\depth_AM=\dim
A$. Note that, in this circumstance, $\depth_AM$ equals the common
length of the maximal $M$-regular sequences in the maximal ideal of
$A$. In \cite[Corollary 2.5]{SSS}, it is shown that $A\bowtie^f J$
is Cohen-Macaulay if and only if $A$ is Cohen-Macaulay and $J$ is a
maximal Cohen-Macaulay $A$-module. Our first main result improves
this corollary by removing the Noetherian assumption.

The reader should be aware that when we say $A\bowtie^f J$ is
Cohen-Macaulay in the sense of a non-empty class of ideals, we mean
$A\bowtie^f J$ is Cohen-Macaulay as a ring.

\begin{thm}\label{m}
Assume that $(A,\mm)$ is quasi-local such that $\mm$ is finitely
generated. Assume that $J$ is contained in the Jacobson radical of
$B$ and it is finitely generated as an $A$-module. Then $A\bowtie^f
J$ is Cohen-Macaulay (ring) in the sense of maximal ideals if and
only if $A$ is Cohen-Macaulay in the sense of maximal ideals and
$\kgr_A(\mm,J)=\dim A$.
\end{thm}
\begin{proof}
Assume that $\mm$ is generated by the sequence ${\bf
a}=a_1,\ldots,a_n$ and that $J$ is generated by the sequence
$\bb=b_1,\ldots,b_m$. Hence $\mm^{\prime_f}=\mm\bowtie^f J$, the
unique maximal ideal of $A\bowtie^f J$ \cite[Corollary
2.7(3)]{DFF1}, is generated by the sequence ${\bf
c}=(a_1,f(a_1)),\ldots, (a_n,f(a_n)),(0,b_1),\ldots,(0,b_m)$.
Notice, by \cite[Corollary 3.2 and Remark 3.3]{DFF1}, that one has
$\sqrt{\iota_A(\ab)(A\bowtie^f J)}=\sqrt{\mm(A\bowtie^f
J)}=\mm^{\prime_f}={\bf c}(A\bowtie^f J)$. Therefore
\begin{align*}
   \kgr_{A\bowtie^f J}(\mm^{\prime_f},A\bowtie^f J)
    &=\chgr_{A\bowtie^f J}(\mm^{\prime_f},A\bowtie^f J)\\
    & = \inf \{i|H^i_{\bf c}(A\bowtie^f J)  \neq 0\}\\
    & = \inf \{i|H^i_{\iota_A(\ab)}(A\bowtie^f J)  \neq 0\} \\
    & = \inf \{i|H^i_{\ab}(A\bowtie^f J)  \neq 0\} \\
    & = \inf \{i|H^i_{\ab}(A)\oplus H^i_{\ab}(J)  \neq 0\} \\
    & = \min \{\chgr_A(\mm,A),\chgr_A(\mm,J)\}\\
    & = \min \{\kgr_A(\mm,A),\kgr_A(\mm,J)\}.
\end{align*}
The first equality obtains by \cite[Proposition 2.7]{HM}, the third
equality follows from \cite[Proposition 2.1(e)]{HM} in conjunction
with $\sqrt{\iota_A(\ab)(A\bowtie^f J)}={\bf c}(A\bowtie^f J)$, the
forth equality deduces from \cite[Proposition 2.1(f)]{HM}, and the
fifth equality holds since as an $A$-module $A\bowtie^f J\cong
A\oplus J$ \cite[Lemma 2.3(4)]{DFF}.

Consequently, the conclusion yields by the equality
$$\kgr_{A\bowtie^f J}(\mm^{\prime_f},A\bowtie^f J)=\min \{\kgr_A(\mm,A),\kgr_A(\mm,J)\}$$
together with $\dim A\bowtie^f J=\dim A$. This last equality holds
true, since $A\bowtie^f J$ is integral over $A$ (see
\cite[Proposition 4.2]{DFF2}).
\end{proof}

\begin{cor} (See \cite[Corollary 2.5]{SSS})
Assume that $A$ is Noetherian local, and that $J$ is contained in
the Jacobson radical of $B$ and it is finitely generated as an
$A$-module. Then $A\bowtie^f J$ is Cohen-Macaulay (ring) if and only
if $A$ is Cohen-Macaulay and $J$ is a maximal Cohen-Macaulay
$A$-module.
\end{cor}

The key to the next theorem is given by the following elementary
lemmas. Their proofs are straightforward; so that we omit them.
Recall from \cite[Corollary 2.5]{DFF1} that the prime ideals of
$A\bowtie^fJ$ are of the type $\overline{\qq}^f$ or
$\pp^{\prime_f}$, for $\qq$ varying in $\spec(B)\backslash V(J)$ and
$\pp$ in $\spec(A)$, where
\begin{align*}
\pp^{\prime_f}:= & \pp\bowtie^fJ:=\{(p,f(p)+j)|p\in \pp, j\in J\}, \\[1ex]
\overline{\qq}^f:= & \{(a,f(a)+j)|a\in A, j\in J, f(a)+j\in \qq\}.
\end{align*}

\begin{lem}\label{e}
Assume that $\ka$ is an ideal of $A$, $\pp$ is a prime ideal of $A$
and that $\qq$ is a prime ideal of $B$. Then
\begin{enumerate}
  \item $\ka^e\subseteq\pp^{\prime_f}$ if and only if $\ka\subseteq\pp$.
  \item $\ka^e\subseteq\bar{\qq}^{f}$ if and only if $f(\ka)\subseteq\qq$.
\end{enumerate}
\end{lem}

In the sequel, we use $\nil(B)$ to denote the nil radical of the
ring $B$.

\begin{lem}\label{ht}
Assume that $\ka$ is an ideal of $A$, $J\subseteq \nil(B)$ and that
$\pp$ is a prime ideal of $A$. Then
\begin{enumerate}
  \item $\pp\in\Min(\ka)$ if and only if $\pp^{\prime_f}\in\Min(\ka^e)$.
  \item $\height\ka=\height\ka^e$.
  \item $\Min(\pp^e)=\{\pp^{\prime_f}\}$.
  In particular $\height\pp^e=\height\pp^{\prime_f}$.
\end{enumerate}
\end{lem}

\begin{prop}\label{j}
Let $\mathcal{A}$ be a non-empty class of ideals of $A$. Assume that
$\height\ka^e \ge \height \ka$ for each $\ka\in \mathcal{A}$. If
$A\bowtie^f J$ is Cohen-Macaulay (ring) in the sense of
$\mathcal{A}^e:=\{\ka^e|\ka\in\mathcal{A}\}$, then $A$ is
Cohen-Macaulay in the sense of $\mathcal{A}$ and $\kgr_A(\ka,J)\ge
\height\ka$ for each $\ka\in \mathcal{A}$.
\end{prop}
\begin{proof}
Assume that $\ka\in \mathcal{A}$. Then, by Lemma \ref{kgr1}(2), we
have
\begin{align*}
    \kgr_A(\ka,A)
    &\geq \kgr_{A\bowtie^f J}(\ka^e,A\bowtie^f J)\\
    &=\height\ka^e\\
    &\geq\height \ka\\
    &\ge \kgr_A(\ka,A).
\end{align*}
Thus $\kgr_A(\ka,A)=\height \ka$. This means that $A$ is
Cohen-Macaulay in the sense of $\mathcal{A}$. Similarly, one obtains
$\kgr_A(\ka,J)\ge \height\ka$.
\end{proof}

It is not clear for us whether, in general, the inequality
$\height\ka^e \ge \height \ka$ holds for each $\ka\in \mathcal{A}$.
However, under the assumption $J\subseteq \nil(B)$, for each ideal
$\ka$, one has the equality $\height\ka^e=\height \ka$ by Lemma
\ref{ht}.

The second main result of the paper is the following theorem.

\begin{thm}\label{jzirenil}
Assume that $J\subseteq \nil (B)$. Then
  $A\bowtie^f J$ is Cohen-Macaulay (ring) in the sense of ideals
if and only if $A$ is Cohen-Macaulay in the sense of ideals and
$\kgr_A(\ka,J)\ge \height\ka$ for every ideal $\ka$ of $A$.
\end{thm}
\begin{proof}
One implication follows from Proposition \ref{j} and Lemma
\ref{ht}(2). Then, to prove the converse, assume that $A$ is
Cohen-Macaulay in the sense of ideals and $\kgr_A(\ka,J)\ge
\height\ka$ for every ideal $\ka$ of $A$. Let $\ka$ be an ideal of
$A$. First observe, by Lemmas \ref{kgr} and \ref{ht}(2), that
\begin{align*}
    \kgr_{A\bowtie^f J}(\ka^e,A\bowtie^f J)
    &=\kgr_A (\ka,A)\\
    &=\height\ka\\
    &=\height\ka^e.
\end{align*}
Now, let $I$ be an arbitrary proper ideal of $A\bowtie^f J$. Then,
by \cite[Theorem 16 of Chapter 5]{N}, there exists a prime ideal
$\mathcal{P}$ of $A\bowtie^f J$ containing $I$ such that
$\kgr_{A\bowtie^f J}(I,A\bowtie^f J)=\kgr_{A\bowtie^f
J}(\mathcal{P},A\bowtie^f J)$. Notice that
$\mathcal{P}=\pp^{\prime_f}$ for some prime ideal $\pp$ of $A$ by
\cite[Corollaries 2.5 and 2.7]{DFF1}. Hence, by Lemma \ref{ht}(3),
one has
\begin{align*}
    \height I
    &\geq \kgr_{A\bowtie^f J}(I,A\bowtie^f J)\\
    &=\kgr_{A\bowtie^f J}(\pp^{\prime_f},A\bowtie^f J)\\
    &\geq \kgr_{A\bowtie^f J}(\pp^e,A\bowtie^f J)\\
    &= \height\pp^e\\
    &= \height\pp^{\prime_f}\\
    &\geq \height I.
\end{align*}
Therefore $A\bowtie^f J$ is Cohen-Macaulay in the sense of ideals.
\end{proof}

The next example shows that, if, in the above theorem, the
hypothesis $J\subseteq \nil(B)$ is dropped, then the corresponding
statement is no longer always true.

\begin{ex}\label{ex}
Let $k$ be a field and $X,Y$ are algebraically independent
indeterminates over $k$. Set $A:=k[[X]]$, $B:=k[[X,Y]]$ and let
$J:=(X,Y)$. Let $f: A\to B$ be the inclusion. Note that $A$ is
Cohen-Macaulay and $\kgr_A(\ka,J)=\height \ka$ for every ideal $\ka$
of $A$. Indeed, if $\ka$ is a non-zero  proper ideal of $A$, and $a$
is a non-zero element of $\ka$, then one has
    $$1\leq\kgr_A(aA,J)\leq\kgr_A(\ka,J)\leq\height_J\ka\leq\height\ka\leq1.$$
The first and second inequalities follow from \cite[Proposition
9.1.2(a),(f)]{BH98}, respectively. While the third inequality
follows from Lemma \ref{tamime 3.2}(ii), the others are obvious.
However, $A\bowtie^f J$ which is isomorphic to $k[[X,
Y,Z]]/(Y,Z)\cap(X-Y )$ is not Cohen-Macaulay.
\end{ex}

Let $M$ be a $A$-module. Then $A\ltimes M$ denotes the \emph{trivial
extension} of $A$ by $M$. It should be noted that $0\ltimes M$ is an
ideal in $A\ltimes M$ and $(0\ltimes M)^2=0$. As in \cite[Example
2.8]{DFF}, if $B:=A\ltimes M$, $J:=0\ltimes M$, and $f:A\to B$ be
the natural embedding, then $A\bowtie^f J\cong A\ltimes M$. Hence
the next result follows from Theorem \ref{jzirenil}. With it, we not
only offer an application of Theorem \ref{jzirenil}, but we also
provide a generalization of the well-known characterization of when
the trivial extension is Cohen-Macaulay in the Noetherian (local)
case, see \cite[Corollary 4.14]{AW}.

\begin{cor}\label{t}
Let $M$ be an $A$-module. Then $A\ltimes M$ is Cohen-Macaulay (ring)
in the sense of ideals if and only if $A$ is  Cohen-Macaulay in the
sense of ideals and $\kgr_A (\ka,M)\ge \height \ka$ for every ideal
$\ka$ of $A$.
\end{cor}

Assume that $A$ is Noetherian. In \cite[Corollary 2.7]{SSS}, the
authors showed that $A$ is Cohen-Macaulay if $A\bowtie^f J$ is
Cohen-Macaulay provided that $f^{-1}(\qq)\neq \mm$ for each
$\qq\in\spec(B)\backslash V(J)$ and each $\mm\in\Max(A)$. In the
following corollary we improve the conclusion of the mentioned
result in the circumstance that $J\subseteq \nil(B)$.

Assume that $A$ is Noetherian and $M$ is a finitely generated
$A$-module. It can be seen that
$\height\ka\leq\grade_A(\ka,M)(=\kgr_A(\ka,M))$ for every ideal
$\ka$ of $A$ if and only if $M_{\pp}$ is maximal Cohen-Macaulay for
every prime ideal $\pp\in\Supp_A(M)$. Indeed, assume that $M_{\pp}$
is maximal Cohen-Macaulay for every prime ideal $\pp\in\Supp_A(M)$,
and $\ka$ is an ideal of $R$. There is nothing to prove if $\ka
M=M$, since in this case $\grade_A(\ka,M)=\infty$. So assume that
$\ka M\neq M$. Then using \cite[Proposition 1.2.10(a)]{BH98}, there
is a prime ideal $\pp$ containing $\ka$ such that
$\grade_A(\ka,M)=\depth M_{\pp}$. Hence by assumption one has
$\grade_A(\ka,M)=\depth M_{\pp}=\dim R_{\pp}=\height \pp\geq \height
\ka$. To prove the converse assume that $\pp\in\Supp_A(M)$. Then
again in view of \cite[Proposition 1.2.10(a)]{BH98}, one has $\dim
R_{\pp}=\height \pp\leq\grade_A(\pp,M)\leq\depth M_{\pp}$. Thus
$M_{\pp}$ is maximal Cohen-Macaulay.

\begin{cor}
Assume that $A$ is Noetherian, and that $J\subseteq \nil B$ is
finitely generated as an $A$-module. Then $A\bowtie^f J$ is
Cohen-Macaulay if and only if $A$ is Cohen-Macaulay and $J_{\pp}$ is
maximal Cohen-Macaulay for every prime ideal $\pp\in\Supp_A(J)$.
\end{cor}

The next proposition provides other sufficient and necessary
condition for $A\bowtie^f J$ to be Cohen-Macaulay in the sense of
ideals.

\begin{prop}\label{int}
With the notation and hypotheses of the beginning of Section 3, one
has
\begin{enumerate}
  \item Let $\mathcal{A}$ be a non-empty class of ideals of $A$. Assume that $\height f^{-1}(\qq)\leq\height\qq$ for every
  $\qq\in\spec(B)\backslash V(J)$. If $A\bowtie^f J$ is Cohen-Macaulay (ring) in the sense of $\mathcal{A}^e:=\{\ka^e|\ka\in\mathcal{A}\}$, then $A$ is  Cohen-Macaulay in the sense of $\mathcal{A}$ and $\kgr_A
(\ka,J)\ge \height \ka$ for every $\ka\in\mathcal{A}$.
  \item Assume that $\height \mathcal{P}\leq\height \mathcal{P}^c$ for
  every $\mathcal{P}\in\spec(A\bowtie^f J)$, where the contraction $\mathcal{P}^c$
  is given with respect to $\iota_A$. If $A$ is  Cohen-Macaulay in the sense of ideals and $\kgr_A
(\ka,J)\ge \height \ka$ for every ideal $\ka$ of $A$, then
$A\bowtie^f J$ is Cohen-Macaulay (ring) in the sense of ideals.
\end{enumerate}
\end{prop}
\begin{proof}
(1) Assume that $A\bowtie^f J$ is Cohen-Macaulay ring in the sense
of $\mathcal{A}^e$. In order to prove the assertion, by Proposition
\ref{j}, it is enough for us to show that $\height\ka^e \ge \height
\ka$ for each ideal $\ka\in\mathcal{A}$. To this end, assume that
$\ka\in\mathcal{A}$ and that $\mathcal{P}$ is a prime ideal of
$A\bowtie^f J$ containing $\ka^e$. In view of \cite[Corollaries 2.5
and 2.7]{DFF1}, one has the following three cases to consider.

\textbf{Case 1.} If $\Pc=\pp^{\prime_f}$ for some prime ideal $\pp$
of $A$ such that $f^{-1}(J)\nsubseteq\pp$, then
$$
\height\mathcal{P}=\height\pp^{\prime_f}=\dim(A\bowtie^f
J)_{\pp^{\prime_f}}=\dim A_{\pp}=\height\pp\geq\height\ka,
$$
by \cite[Proposition 2.9]{DFF1} and Lemma \ref{e}(1).

\textbf{Case 2.} If $\Pc=\pp^{\prime_f}$ for some prime ideal $\pp$
of $A$ such that $f^{-1}(J)\subseteq\pp$, then
\begin{align*}
    \height \mathcal{P}
    &= \height\pp^{\prime_f}\\
    &=\dim(A\bowtie^f J)_{\pp^{\prime_f}}\\
    &=\dim(A_{\pp}\bowtie^{f_{\pp}} J_{S_{\pp}})\\
    &=\max\{\dim A_{\pp},\dim  (f_\pp(A_\pp)+J_{S_\pp})\}\\
    &\geq \dim A_{\pp}\\
    &= \height \pp\\
    &\geq\height \ka,
\end{align*}
by \cite[Proposition 2.9]{DFF1}, \cite[Proposition 4.1]{DFF2} and
Lemma \ref{e}(1), where $S_\pp:=f(A\backslash\pp)+J$.

\textbf{Case 3.} If $\Pc=\bar{\qq}^f$ for some prime ideal $\qq$ of
$B$, then
\begin{align*}
    \height \mathcal{P}
    &=\height \bar{\qq}^f\\
    &= \dim(A\bowtie^fJ)_{\bar{\qq}^f}\\
    &=\dim B_{\qq}\\
    &=\height {\qq}\\
    &\geq \height f^{-1}(\qq)\\
    &\geq \height \ka.
\end{align*}
The third equality follows by \cite[Proposition 2.9]{DFF1}, the
first inequality holds by assumption, and the second one follows by
Lemma \ref{e}. This completes the proof of the first assertion.

(2) Assume that $A$ is  Cohen-Macaulay in the sense of ideals and
that $\kgr_A(\ka,J)\ge \height \ka$ for every ideal $\ka$ of $A$. As
indicated by \cite[Theorem 3.3]{AT}, it is enough to show that
$$\kgr_{A\bowtie^f J}(\Pc,A\bowtie^f J)= \height \Pc$$
for every prime ideal $\Pc$ of $A\bowtie^f J$. Let $\Pc$ be a prime
ideal of $A\bowtie^f J$. Then
\begin{align*}
    \height \Pc
    &\leq \height\Pc^c\\
    &=\kgr_A(\Pc^c,A)\\
    &=\kgr_{A\bowtie^f J}(\Pc^{ce},A\bowtie^f J)\\
    &\le \kgr_{A\bowtie^f J}(\Pc,A\bowtie^f J)\\
    &\le \height \Pc.
\end{align*}
The first inequality holds by assumption, the second inequality is
by \cite[Proposition 9.1.2(f)]{BH98}, and the last one is by Lemma
\ref{tamime 3.2}(2), and the second equality follows from Lemma
\ref{kgr}.
\end{proof}

We are now in a position to present our third main result.

\begin{thm}\label{gd}
With the notation and hypotheses of the beginning of Section 3, the
following statements hold:
\begin{enumerate}
  \item Let $\mathcal{A}$ be a non-empty class of ideals of $A$. Assume that the homomorphism $f:A\to B$ satisfies the going-down
property. If $A\bowtie^f J$ is Cohen-Macaulay (ring) in the sense of
$\mathcal{A}^e:=\{\ka^e|\ka\in\mathcal{A}\}$, then $A$ is
Cohen-Macaulay in the sense of $\mathcal{A}$ and $\kgr_A (\ka,J)\ge
\height \ka$ for every $\ka\in\mathcal{A}$.
  \item Assume that $\iota _A: A\to A\bowtie^f J$ is an integral ring
extension. If $A$ is  Cohen-Macaulay in the sense of ideals and
$\kgr_A (\ka,J)\ge \height \ka$ for every ideal $\ka$ of $A$, then
$A\bowtie^f J$ is Cohen-Macaulay (ring) in the sense of ideals.
\end{enumerate}
\end{thm}
\begin{proof}
It is well-known that $\height f^{-1}(\qq)\leq\height\qq$ for every
$\qq\in\spec(B)$ if the homomorphism $f:A\to B$ satisfies the
going-down property by \cite[Exercise 9.9]{M}. In the light of
Proposition \ref{int}, this proves (1). To prove (2), keeping in
mind Proposition \ref{int}, notice that, for every
$\mathcal{P}\in\spec(A\bowtie^f J)$, the inequality $\height
\mathcal{P}\leq\height \mathcal{P}^c$ holds since $\iota _A: A\to
A\bowtie^f J$ is an integral ring extension \cite[Exercise 9.8]{M},
where the contraction $\mathcal{P}^c$ is given with respect to
$\iota_A$.
\end{proof}

Note that Example \ref{ex} also shows that we can not neglect the
integral assumption in part two of the above theorem.

\begin{ex}
\begin{enumerate}
\item Assume that $A$ is an integral domain with $\dim A\leq1$
and that $B$ is an integral domain containing $A$. Assume that $J$
is an ideal of $B$ which is finitely generated $A$-module. Hence, as
in Example \ref{ex}, one has $\kgr_A(\ka,J)=\height \ka$ for every
proper ideal $\ka$ of $A$. Notice that $A$ is Cohen-Macaulay in the
sense
    of ideals by \cite[Page 2305]{AT}.
    Therefore one obtains that $A\bowtie^f J$ is Cohen-Macaulay
    in the sense of ideals by Theorem \ref{gd}.
\item To construct a concrete example for (1), set $A:=\mathbb{Q}+X\mathbb{R}[X]$,
where $\mathbb{Q}$ is the field of rational numbers, $\mathbb{R}$ is
the field of real numbers and $X$ is an indeterminate over
$\mathbb{R}$. It is easy to see that $A$ is a one dimensional non
integrally closed domain. Put $B:=A[\sqrt{2}]$, which is finitely
generated as an $A$-module. Let $J$ be a finitely generated ideal of
$B$. Consequently, by (1), $A\bowtie^f J$ is Cohen-Macaulay in the
sense of ideals.
\item Assume that $A$ is a valuation domain, $B$ an arbitrary
integral domain containing $A$ and that $J$ is an ideal of $B$. Then
by \cite[Corollary 4]{D1} and \cite[Theorem 1]{D2}, the inclusion
homomorphism $f:A\hookrightarrow B$ satisfies the going-down
property. Also notice, by \cite[Proposition 3.12]{AT}, that $A$ is
Cohen-Macaulay in the sense of ideals if and only if $\dim A\leq1$.
Further, assume that $\dim A>1$. Then $A\bowtie^f J$ can never be
Cohen-Macaulay in the sense of ideals by Theorem \ref{gd}. In
particular, the composite ring extensions $A+XB[X]$ and $A+XB[[X]]$
can never be Cohen-Macaulay in the sense of ideals.
\end{enumerate}
\end{ex}

Note that if $J$ is finitely generated as an $A$-module, then $\iota
_A: A\to A\bowtie^f J$ is an integral ring extension, and that, in
this case, $\kgr_A(\ka, J)\le \height\ka$ for every ideal $\ka$ of
$A$ by Lemma \ref{tamime 3.2}. Hence we can make the following
corollaries right away.

\begin{cor}
Assume that the homomorphism $f:A\to B$ satisfies the going-down
property and that $J$ is finitely generated as an $A$-module. Then
$A\bowtie^f J$ is Cohen-Macaulay (ring) in the sense of ideals if
and only if $A$ is  Cohen-Macaulay in the sense of ideals and
$\kgr_A (\ka,J)= \height \ka$ for every ideal $\ka$ of $A$.
\end{cor}

\begin{cor}\label{int1}
Assume that $f:A\to B$ is a monomorphism of integral domains, and
$A$ is integrally closed and that $B$ is integral over $A$. Then
$A\bowtie^f J$ is Cohen-Macaulay (ring) in the sense of ideals if
and only if $A$ is Cohen-Macaulay in the sense of ideals and $\kgr_A
(\ka,J)\ge \height \ka$ for every ideal $\ka$ of $A$.
\end{cor}
\begin{proof}
By \cite[Theorem 9.4]{M}, $f:A\to B$ satisfies the going-down
property. Also, $\iota _A: A\to A\bowtie^f J$ is an integral ring
extension by assumption and \cite[Lemma 3.6]{DFF2}.
\end{proof}

\begin{cor}\label{int2}
Assume that $f:A\to B$ is a flat and integral homomorphism. Then
$A\bowtie^f J$ is Cohen-Macaulay (ring) in the sense of ideals if
and only if $A$ is Cohen-Macaulay in the sense of ideals and
$\kgr_A(\ka,J)\ge \height \ka$ for every ideal $\ka$ of $A$.
\end{cor}
\begin{proof}
By \cite[Theorem 9.5]{M}, $f:A\to B$ satisfies the going-down
property. Also, $\iota _A: A\to A\bowtie^f J$ is an integral ring
extension by assumption and \cite[Lemma 3.6]{DFF2}.
\end{proof}

In concluding, we apply Corollary \ref{int2} on amalgamated
duplication. Recall that if $f:=id_A$ is the identity homomorphism
on $A$, and $J$ is an ideal of $A$, then $A\bowtie
J:=A\bowtie^{id_A} J$ is called the amalgamated duplication of $A$
along $J$. Assume that $(A,\mm)$ is Noetherian local. In
\cite[Discussion 10]{D}, assuming that $A$ is Cohen-Macaulay, D'Anna
showed that $A\bowtie J$ is Cohen-Macaulay if and only if $J$ is
maximal Cohen-Macaulay. Next in \cite[Corollary 2.7]{SSh}, the
authors improved D'Anna's result as $A\bowtie J$ is Cohen-Macaulay
if and only if $A$ is Cohen-Macaulay and $J$ is maximal
Cohen-Macaulay. Our final corollary generalizes these results.

\begin{cor}\label{d}
Let $J$ be an ideal of $A$. Then $A\bowtie J$ is Cohen-Macaulay
(ring) in the sense of ideals if and only if $A$ is Cohen-Macaulay
in the sense of ideals and $\kgr_A(\ka,J)\ge \height \ka$ for every
ideal $\ka$ of $A$.
\end{cor}
\begin{proof}
This immediately follows from Corollary \ref{int2}, since
$f=id_A:A\to A$ is flat and integral.
\end{proof}

 {\bf Acknowledgements.}
The authors is deeply grateful to the referee for a very careful
reading of the manuscript and many valuable suggestions.

\end{document}